\documentclass{article}

\PassOptionsToPackage{numbers,sort&compress}{natbib}
\PassOptionsToPackage{dvipsnames}{xcolor}

\usepackage[preprint]{neurips_2026}

\usepackage[utf8]{inputenc} %
\usepackage[T1]{fontenc}    %
\usepackage{hyperref}       %
\usepackage{url}            %
\usepackage{booktabs}       %
\usepackage{amsfonts}       %
\usepackage{nicefrac}       %
\usepackage{microtype}      %
\usepackage{xcolor}

\usepackage{amsmath,amssymb,amsthm,mathtools}
\usepackage{enumitem}
\usepackage{graphicx}
\usepackage{authblk}
\usepackage{aliascnt}
\usepackage[capitalize,noabbrev]{cleveref}

\newtheorem{theorem}{Theorem}
\newcommand{\sharedthm}[2]{%
  \newaliascnt{#1}{theorem}\newtheorem{#1}[#1]{#2}\aliascntresetthe{#1}}
\sharedthm{proposition}{Proposition}
\sharedthm{corollary}{Corollary}
\sharedthm{lemma}{Lemma}
\theoremstyle{definition}
\sharedthm{remark}{Remark}
\sharedthm{example}{Example}

\newcommand{\R}{\mathbb{R}}
\newcommand{\Sph}{\mathbb{S}}

\newcommand{\cB}{\mathcal{B}}

\hypersetup{colorlinks=true,linkcolor=MidnightBlue,citecolor=OliveGreen,urlcolor=MidnightBlue}

\title{Non-Euclidean Gradient Descent \\  Has Many Edges of Stability}

\author{\makebox[10.5em]{Peter Zaika\rlap{\textsuperscript{$*$}}}}
\author{\makebox[10.5em]{Zakhar Shumaylov\rlap{\textsuperscript{$*$}}}}
\affil{University of Cambridge\protect\\[2pt] {\small \mbox{\llap{$^*$}Equal contribution}}}
\date{}

\begin{document}
\maketitle

\begin{abstract}
Recent observations that deep learning models train at the edge of stability have renewed interest in understanding stability properties of classical optimization algorithms. For quadratic objectives, standard gradient descent has a sharp stability threshold: it converges from every initialization when $\eta<2/L$, diverges from almost every initialization when $\eta>2/L$, and diverges from every nonzero initialization when $\eta>2/\mu$. For non-Euclidean gradient descent, the corresponding known guarantees are weaker: convergence holds for every initialization when $\eta<2/S$, where $S$ is the generalized sharpness, while every nonzero initialization diverges when $\eta>2/\mu$. The behavior in the intermediate regime $2/S<\eta<2/\mu$ is not fully characterized. We show that Hilbertian norms are the only ones for which $2/S$ is an almost-everywhere sharp threshold for every positive-definite quadratic objective. Using linear conjugacy between non-Euclidean and Euclidean gradient descent, we show that for general norms the intermediate regime can contain arbitrarily many edges of stability associated with distinct attracting directions. We identify a forward-invariant region of exponentially divergent directions beyond the generalized stability threshold. When the directions of minimum curvature are finite and nondegenerate, we prove that, sufficiently close to $2/\mu$ from below, only their invariant directions remain convergent. We give smooth two-dimensional constructions with arbitrarily many distinct thresholds, arbitrarily large relative threshold separation and show these properties to be stable.
\end{abstract}

\begin{figure}[!ht]
\centering
\includegraphics[width=\linewidth]{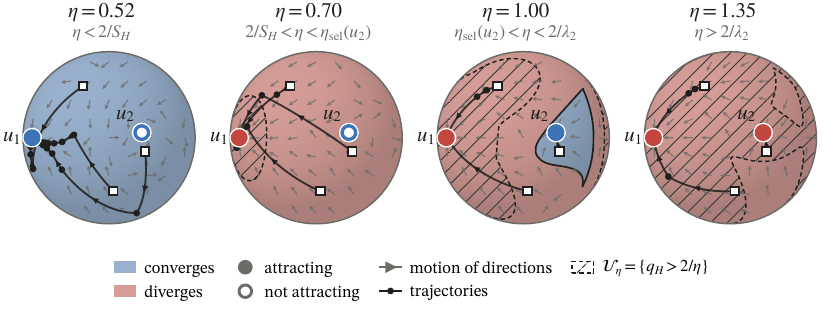}
\caption{Convergent (blue) and divergent (red) initial directions for the example of \cref{sec:3d} at four step sizes, on one hemisphere of the conjugate $z$-coordinates, which contains every direction since $u$ and $-u$ coincide. Markers show the local maxima $u_1$ (global) and $u_2$ of $q_H$: filled if they attract nearby directions, and coloured by the fate of their own line. Grey arrows show the two-step direction of motion; black curves are three trajectories of directions, starting at the open squares. Hatching marks the invariant divergence region $\mathcal U_\eta$ of \cref{thm:escape}. Computational details are in \cref{app:3d}.}
\label{fig:shells3d}
\end{figure}

\section{Introduction}\label{sec:intro}

The edge-of-stability (EoS) phenomenon was first documented for vanilla gradient descent in deep learning by \citet{Cohen2021} and related dynamical-stability work of \citet{Wu2018}: during training, the top Hessian eigenvalue grows toward and then remains near the stability scale $2/\eta$. Subsequent work has connected this regime to bifurcation structure \cite{SongYun2023}, cubic self-stabilization \cite{Damian2023}, and low-dimensional oscillatory dynamics beyond the stability plateau \cite{Ghosh2025}. For preconditioned and adaptive methods, analogous edge-of-stability descriptions involve the spectrum of the preconditioned Hessian \cite{Cohen2022}. \citet{islamov2026noneuclideangradientdescentoperates} ask how this picture extends beyond inner-product geometries.

\noindent The classical edge of stability is easiest to introduce on a quadratic objective:
\[
L(x)=\frac12x^\top Hx,\qquad H\succ0.
\]
For Euclidean gradient descent, the iteration diagonalizes in the eigenspaces of $H$. If
\[
S_H=\lambda_{\max}(H),
\qquad
\mu_H=\lambda_{\min}(H),
\]
then every initialization converges for $\eta<2/S_H$, almost every initialization diverges for $\eta>2/S_H$, and every nonzero initialization diverges for $\eta>2/\mu_H$. The first threshold is therefore an almost-everywhere sharp edge; the second removes the remaining lower-eigenvalue stable subspaces. This local quadratic mechanism is the standard explanation underlying the Euclidean EoS picture \cite{Cohen2021,Damian2023}.

\noindent \citet{islamov2026noneuclideangradientdescentoperates} ask how this picture extends to
non-Euclidean gradient descent.  For a general norm $\|\cdot\|$, they study the update
\begin{equation}
 x_{t+1}=x_t-\eta\nabla\Phi^*(Hx_t),
 \qquad
 \Phi(x)=\frac12\|x\|^2,
 \quad \Phi^*(x)
:= \frac12\|x\|_*^2.
 \label{eq:ngd}
\end{equation}
and introduce the generalized sharpness, motivated through \emph{directional smoothness}, following \citet{Mishkin2024}, as the curvature actually experienced along a single optimization step. 
\begin{equation}
 S_H:=\max_{\|d\|=1} d^\top H d,
 \qquad \mu_H:=\min_{\|d\|=1}d^\top H d.
 \label{eq:Sintro}
\end{equation}
In what follows, $\Phi$ is assumed to be $C^2$ on $\R^d\setminus\{0\}$ and $\nabla^2\Phi(x)\succ0$ for every $x\neq0$. These assumptions make the duality map unique and smooth away from the origin.
Their exact loss-change identity shows that, whenever the gradient is nonzero, the loss decreases exactly when directional smoothness lies below $2/\eta$; generalized sharpness is then the corresponding worst-direction Hessian quantity \cite[Secs.~2--3]{islamov2026noneuclideangradientdescentoperates}.
For positive-definite quadratics, \citet[Sec.~5]{islamov2026noneuclideangradientdescentoperates} prove two complementary results. Their Theorem~5.1 gives
\begin{equation}
\eta<\frac2{S_H}
\quad\Longrightarrow\quad
\text{linear convergence from every initialization}.
\label{eq:ICCGstable}
\end{equation}
Their Theorem~5.2 shows that the threshold cannot be improved as a \emph{uniform} guarantee: for every $\eta>2/S_H$, there exists an initialization, together with a valid dual-gradient selection when the latter is nonunique, from which the method diverges. The key is their Lemma~5.3. If $\hat d$ maximizes \eqref{eq:Sintro}, then $\hat d$ is preserved by the update and, along its span,
\begin{equation}
x_t=(1-\eta S_H)^t x_0.
\label{eq:ICCGline}
\end{equation}
Hence the sharpest direction becomes unstable exactly at $2/S_H$. This is much weaker than the Euclidean conclusion: it only gives divergence on a particular invariant line, whereas Euclidean GD diverges from all but a zero-measure set of initializations \cite[Sec.~5]{islamov2026noneuclideangradientdescentoperates}. Despite this, quadratic experiments suggest that divergence is much more generic than the theorem proves, and they explicitly leave this gap open. 
\\\;\\
\noindent This paper resolves the quadratic gap in a different direction than one might first expect.  In general, the Euclidean almost-everywhere statement is \emph{false}: beyond $2/S_H$ there may simultaneously be open sets of divergent and convergent initializations. 
\begin{figure}[t!]
\centering
\includegraphics[width=\linewidth]{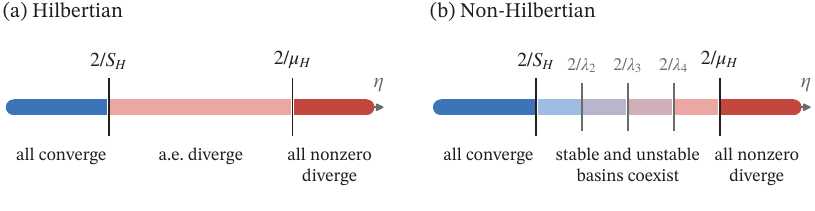}
\caption{Stability regimes of non-Euclidean gradient descent on a positive-definite quadratic, as a function of the step size $\eta$ (positions schematic). For every norm, all initializations converge for $\eta < 2/S_H$ and every nonzero initialization diverges for $\eta > 2/\mu_H$. \textbf{(a)}~For a Hilbertian norm, almost every initialization diverges throughout the intermediate window $2/S_H < \eta < 2/\mu_H$. \textbf{(b)}~For non-Hilbertian norms and suitable $H$, different projectively attracting critical directions have distinct radial stability thresholds $2/\lambda_i$, allowing open convergent and divergent basins to coexist.
}
\label{fig:phase}
\end{figure}

\paragraph{First result: Sharp generalized EoS characterizes Hilbertian norms.}
A norm is Hilbertian if $\|x\|^2=x^\top Px$ for some $P\succ0$. Let
$
\cB_\eta
:=\{x_0\in\R^d:x_t\to0
\text{ under \eqref{eq:ngd}}\}
$
denote the convergence basin. We show that the almost-everywhere condition completely characterizes the allowed norms. 
\[
\boxed{
\|\cdot\|\text{ is Hilbertian}
\quad\Longleftrightarrow\quad
\forall H\succ0,\ \forall\eta>2/S_H,\ 
\mu_{\R^d}(\mathcal B_\eta)=0.
}
\]
Equivalently, inner-product norms are exactly the norms for which generalized sharpness retains the Euclidean almost-everywhere sharpness property for \emph{every} positive-definite quadratic. Every non-Hilbertian norm admits a quadratic and a step size $\eta>2/S_H$ for which an open set of initializations still converges (\cref{thm:hilbert}).

This answers the qualitative question left by \citet{islamov2026noneuclideangradientdescentoperates}: generic divergence immediately beyond $2/S_H$ is not a property of arbitrary norm geometry. The rest of the paper explains why.

\paragraph{Second result: Why the single edge splits.}
Set $A=H^{1/2}$ and $z_t=Ax_t$. It is well known that quadratic non-Euclidean GD is linearly conjugate to ordinary Euclidean GD on the convex $2$-homogeneous function \cite{boyd2004convex}. Specifically, \eqref{eq:ngd} is equivalent to:
\begin{equation}
z_{t+1}=z_t-\eta\nabla Q(z_t),
\qquad
Q(z)=\frac12\|Az\|_*^2.
\label{eq:qgd}
\end{equation}
\noindent Since $Q$ is $2$-homogeneous, its dynamics separate into a radial component and a directional component determined by $Q$ on the Euclidean unit sphere. If $u\in\mathbb S^{d-1}$ is a critical point of $Q|_{\mathbb S^{d-1}}$, then $\nabla Q(u)=\lambda u$
for some $\lambda>0$, and the corresponding ray is invariant:
$
z_{t+1}=(1-\eta\lambda)z_t.
$
Thus this direction has radial edge of stability $\eta=\frac2\lambda.$ Defining on the euclidean sphere
\begin{equation}
q_H(u):=2Q(u)=\|Au\|_*^2.
\label{eq:q}
\end{equation}
\noindent In the Euclidean case, the critical directions are exactly the eigendirections of $H$. For a general norm however, $Q|_{\mathbb S^{d-1}}$ may have several strict local maxima with different critical values, producing several distinct stability thresholds. By operator-norm duality, any strict \emph{suboptimal} maximum has:
\begin{equation}
\frac2{S_H}<\frac2\lambda\le\frac2{\mu_H} \qquad\text{as}\qquad
\min_{u\in\Sph^{d-1}}q_H(u)=\mu_H,
\quad
\max_{u\in\Sph^{d-1}}q_H(u)=S_H.
\label{eq:qrange}
\end{equation}
Thus, stable and unstable open basins may coexist at the same step size, meaning that:
\[
\boxed{\frac2{S_H}\text{ is the first edge of stability, not necessarily the only edge.}}
\]
\Cref{fig:phase} summarizes the distinction. For every nondegenerate strict local maximum $u$ of $q_H$, \cref{thm:modal} determines the step sizes at which nearby directions are attracted to $u$, and the rate at which iterates along $u$ grow or decay. A nondegenerate global maximum upgrades the divergent line of \citet[Thm.~5.2]{islamov2026noneuclideangradientdescentoperates} to an open divergent cone, while a nondegenerate suboptimal local maximum gives an open convergent cone for step sizes beyond $2/S_H$. 

Beyond these local results, the radial/projective decomposition also constrains the global stability basins. The high-Rayleigh region $\{q_H>2/\eta\}$ is forward-invariant under gradient descent iterations, and every
trajectory entering it diverges exponentially (\cref{thm:escape}). These regions expand as $\eta$ increases, giving nested upper bounds on the convergence basins, although the basins themselves need not be nested. If the global minima of $q_H$ are isolated and nondegenerate, almost every trajectory already diverges for step sizes sufficiently close to $2/\mu_H$ from below (\cref{thm:terminal}).
In \cref{sec:examples} we then construct an exact smooth two-edge example, an example with arbitrarily many distinct edges of stability in dimension two, show that the splitting into several edges is robust under perturbations, and families in which a stable basin persists arbitrarily far beyond $2/S_H$ in relative terms.

\section{Sharp generalized EoS characterizes Hilbert geometry}\label{sec:hilbert}

We begin with the structural theorem that answers the almost-everywhere question raised by \citet[Sec.~5]{islamov2026noneuclideangradientdescentoperates}. The almost-everywhere sharpness phenomenon over all quadratic objectives characterizes inner-product geometry. Throughout, we consider \(d\geq2\) as the $d = 1$ case is completely trivial.
\begin{theorem}[Hilbertian characterization]
\label{thm:hilbert}
Assume $\Phi(x)=\frac12\|x\|^2$ is $C^2$ on $\R^d\setminus\{0\}$ and $\nabla^2\Phi(x)\succ0$ for all $x\neq0$. The following are equivalent.
\begin{enumerate}[label=(\roman*),leftmargin=*]
\item $\|\cdot\|$ is induced by an inner product.
\item For every $H\succ0$, the generalized Rayleigh quotient
$
R_H(x):={x^\top Hx}/{\|x\|^2}
$
has no strict suboptimal local maxima on $\{\|x\|=1\}$.
\item For every $H\succ0$ and every $\eta>2/S_H$, the convergence basin $\cB_\eta$ has Lebesgue measure zero.
\end{enumerate}
\end{theorem}
\noindent The proof is in \cref{app:hilbert}. The quotient \(R_H\) is convenient for constructing objectives in primal coordinates, whereas \(q_H\) describes the dynamics after conjugacy. \Cref{lem:critical-duality} shows that their critical directions and values correspond.

\section{Radial decomposition}\label{sec:normal}
To understand the non-Eucldiean case, we first relate the critical points of $q_H$ to those of the generalized Rayleigh quotient $R_H(x)=x^\top Hx/\|x\|^2$.
Writing $z_t=r_tu_t$, with $r_t=\|z_t\|_2$ and $u_t\in\Sph^{d-1}$, gives an exact radial/projective decomposition. Proofs for this section are in \cref{app:structural}.

\begin{proposition}[Exact radial/projective dynamics]
\label{prop:radial}
For $u\in\Sph^{d-1}$,
\begin{equation}
 \nabla Q(u)=q_H(u)u+\frac12\nabla_Sq_H(u),
 \label{eq:polargrad}
\end{equation}
where $\nabla_S$ is the spherical gradient. Consequently
\begin{align}
 \frac{r_{t+1}}{r_t}&=\rho_\eta(u_t),\label{eq:rho-main}\\
 u_{t+1}&=
 \frac{(1-\eta q_H(u_t))u_t-\frac\eta2\nabla_Sq_H(u_t)}{\rho_\eta(u_t)},
 \label{eq:projmap}
\end{align}
where $\rho_\eta$ is defined below. The direction update is defined where $\rho_\eta(u)>0$; when $\rho_\eta(u)=0$, the iterate reaches the minimizer in one step.
\begin{equation}
 \rho_\eta(u)^2
 =(1-\eta q_H(u))^2+\frac{\eta^2}{4}\|\nabla_Sq_H(u)\|_2^2.
 \label{eq:rho}
\end{equation}
\end{proposition}

\noindent The critical points of $q_H$ are precisely the preserved projective directions. At such a point $u$ with $q_H(u)=\lambda$,
\[
\nabla Q(u)=\lambda u,
\qquad
z_{t+1}=(1-\eta\lambda)z_t
\]
on the corresponding ray. 
The same decomposition also clarifies the directional smoothness quantity used by \citet[Sec.~2]{islamov2026noneuclideangradientdescentoperates}:
\begin{proposition}[Directional smoothness decomposition]
\label{prop:directional}
Define
\[
D_H(u):=\frac{\|\nabla Q(u)\|_2^2}{q_H(u)}.
\]
This equals the tight quadratic directional smoothness along the non-Euclidean update direction, and
\begin{equation}
D_H(u)=q_H(u)+\frac{\|\nabla_Sq_H(u)\|_2^2}{4q_H(u)}.
\label{eq:Ddecomp}
\end{equation}
In particular $q_H(u)\le D_H(u)\le S_H$, and
\begin{equation}
\rho_\eta(u)^2=1-\eta q_H(u)\bigl(2-\eta D_H(u)\bigr).
\label{eq:rhoD}
\end{equation}
\end{proposition}

\noindent \Cref{eq:Ddecomp} splits directional smoothness into $q_H(u)$, its value when the update preserves the direction of the iterate, and a nonnegative correction that vanishes exactly when it does not rotate it, i.e.\ at critical points of $q_H$.

\section{The edges of stability theorem}\label{sec:modal}
We now linearize the projective dynamics around critical directions to determine when their radial stability extends to nearby initializations. Let $u$ be a critical point of $q_H$ with $\lambda=q_H(u)$, and let $\beta_1,\ldots,\beta_{d-1}$ be the eigenvalues of $\nabla^2Q(u)$ on $u^\perp$. Since $\nabla^2Q(u)=A\nabla^2\Phi^*(Au)A\succ 0,$ all $\beta_i > 0$.
\begin{theorem}[Edges of stability at critical directions]
\label{thm:modal}
Let $u$ be a critical point of $q_H$ with critical value $\lambda$.
\begin{enumerate}[label=(\roman*)]
\item For $\eta \neq 1 / \lambda$, the projective map is well-defined at $[u]$, and its transverse multipliers are
\begin{equation}
 \kappa_i(\eta;u)=\frac{1-\eta\beta_i}{1-\eta\lambda}.
 \label{eq:kappa}
\end{equation}
\item $u$ is a nondegenerate strict local maximum of $q_H$ if and only if $\beta_i<\lambda$ for every $i$.
\item Suppose $u$ is a nondegenerate strict local maximum and let $\beta_{\min}:=\min_i\beta_i$. Then $[u]$ is locally exponentially attracting for
\begin{equation}
 \eta>\eta_{\rm sel}(u):=\frac{2}{\lambda+\beta_{\min}}.
 \label{eq:selection}
\end{equation}
For $0<\eta<\eta_{\rm sel}(u)$, whenever the projective map is defined at $[u]$, the critical direction is linearly unstable. 
\item Let $\eta \neq 1 / \lambda$. For every nonzero initialization whose projective trajectory remains defined and converges to $[u]$,
\begin{equation}
 \lim_{t\to\infty}\frac1t\log L(x_t)
 =2\log|1-\eta\lambda|.
 \label{eq:modegrowth}
\end{equation}
Thus \(2/\lambda\) is the radial stability threshold of the invariant critical line. When \(u\) is a nondegenerate strict local maximum, this threshold lies strictly above its projective selection threshold, so an open attracting cone is radially contracting immediately below \(2/\lambda\) and radially expanding immediately above it.
$
 \eta_{\rm edge}(u)=\frac2\lambda.
 \label{eq:modeedge}
$
\end{enumerate}
\end{theorem}

The theorem gives a mechanism for multiple edges: each nondegenerate strict local maximum of the nonlinear Rayleigh quotient has an attracting projective basin. The proof is in \cref{app:modal}.

\paragraph{Radial versus projective stability.}
The two stability conditions play distinct roles.
At a critical direction $u$ with value $\lambda$,
the radial multiplier $1-\eta\lambda$ determines
whether iterates along the invariant line contract
or expand. The transverse multipliers
$
\kappa_i(\eta;u)
=\frac{1-\eta\beta_i}{1-\eta\lambda}
$
determine whether nearby directions approach
that line in projective space. If $|\kappa_i(\eta;u)|<1$ for every $i$,
then $[u]$ has an open projective basin of
attraction. On the corresponding cone of
initializations, \cref{thm:modal} gives the
asymptotic radial growth rate
$\log|1-\eta\lambda|$.
Consequently, such a cone is exponentially
convergent when $|1-\eta\lambda|<1$, and
exponentially divergent when
$|1-\eta\lambda|>1$.

\noindent Thus, projective attraction allows stability
or instability of an invariant line to extend
to an open set of initializations. In particular,
two projectively attracting critical directions
with different critical values may support
convergent and divergent cones at the same
step size.

\section{Coexisting stable and unstable basins}\label{sec:consequences}
Proofs for this section are in \cref{app:modal}.
The first consequence strengthens the existential divergence statement of \citet[Thm.~5.2]{islamov2026noneuclideangradientdescentoperates} above the generalized threshold.

\begin{corollary}[The first edge has an open divergent basin]
\label{cor:first}
Let \(u_\star\) be a nondegenerate global maximizer of \(q_H\). For every \(\eta>2/S_H\), the projective fixed point \([u_\star]\) is locally exponentially attracting. Its local attracting basin gives rise to an open cone of exponentially divergent initializations.
\end{corollary}

\noindent The second consequence explains exactly how a sharp almost-everywhere edge can fail.

\begin{corollary}[Suboptimal maxima are stable beyond $2/S_H$]
\label{cor:suboptimal}
Suppose $u$ is a nondegenerate strict local maximum of $q_H$ with value $\lambda<S_H$. Then for every
\begin{equation}
 \max\!\left\{\frac2{S_H},\eta_{\rm sel}(u)\right\}
 <\eta<\frac2\lambda,
 \label{eq:stablebeyond}
\end{equation}
there is an open cone of initializations converging exponentially to the origin even though $\eta>2/S_H$.
\end{corollary}

\noindent If $q_H$ has both a suboptimal and a nondegenerate global maximum, stable and unstable basins can therefore coexist.

\begin{figure}[th]
\centering
\includegraphics[width=.72\linewidth]{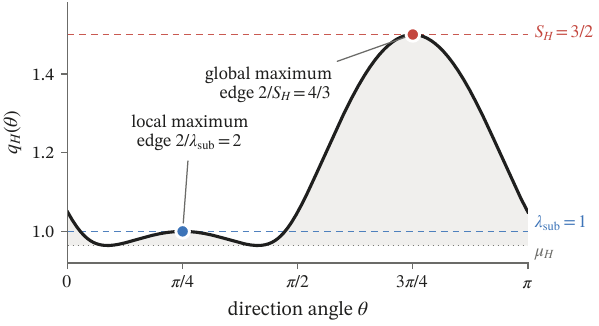}
\caption{The nonlinear Rayleigh landscape $q_H$ on the
Euclidean unit circle for the smooth norm in
\eqref{eq:exactH}. The global and suboptimal maxima are attained
at $\theta=3\pi/4$ and $\theta=\pi/4$, respectively, with
values $S_H=3/2$ and $\lambda_{\rm sub}=1$. Their radial
thresholds are $4/3$ and $2$. The minimum is
$\mu_H=2\sqrt3-5/2$.}
\label{fig:landscape}
\end{figure}
\section{Global geometry of the stability regions}\label{sec:global}
The preceding results describe open basins around particular critical directions, but do not determine their global extent. We now study how the global geometry of these basins is constrained as
the step size varies. Because $Q$ is even, the direction map \eqref{eq:projmap} descends, wherever it is non-zero, to projective space $\mathbb{RP}^{d-1}$. We write $P_\eta$ for this map, and for a subset $E \subseteq \mathbb{RP}^{d-1}$ we write $\operatorname{Cone}_H(E)=\{H^{-1/2}(ru):r\geq0,\ u\in\mathbb S^{d-1},\ [u]\in E\}$ for the associated cone.

\noindent For $\eta>0$ define the high-Rayleigh domain and its eventual-entry set by
\begin{equation}
\mathcal U_\eta:=\{[u]:q_H(u)>2/\eta\},\qquad
\mathcal E_\eta:=\{[u]:P_\eta^n([u])\in\mathcal U_\eta
\text{ for some }n\ge0\},
\label{eq:escape-sets}
\end{equation}
where the iterates in the definition of $\mathcal E_\eta$ are required to be defined through time $n$. For clarity, the direction of a trajectory which reaches zero is not continued beyond that time.

\begin{theorem}[Invariant divergence region]\label{thm:escape}
For every $\eta>0$ the following hold.
\begin{enumerate}[label=(\roman*)]
\item If $[u]\in\mathcal U_\eta$, then $P_\eta([u])$ is defined and
\begin{equation}
q_H(P_\eta([u]))\ge q_H(u),\qquad
\rho_\eta(u)\ge\eta q_H(u)-1>1.
\label{eq:invariant-q}
\end{equation}
In particular, $\mathcal U_\eta$ is forward-invariant.
\item $\mathcal E_\eta$ is open and every nonzero initialization in $\operatorname{Cone}_H(\mathcal E_\eta)$ eventually diverges exponentially. In particular, the convergence basin satisfies
\begin{equation}
\cB_\eta \setminus \{0\}\subseteq
\operatorname{Cone}_H\bigl(\{[u]:q_H(u)\le 2/\eta\}\bigr).
\label{eq:stable-envelope}
\end{equation}
\end{enumerate}
\end{theorem}

\noindent \Cref{thm:escape} implies that every convergent
trajectory must remain in $\{q_H\le 2/\eta\}$.
As $\eta$ approaches $2/\mu_H$ from below,
this region shrinks towards the global minima
of $q_H$. In particular, for every \(\eta>2/S_H\), the set \(\mathcal{U}_\eta\) is nonempty and open, so an open cone of exponentially divergent initializations exists without any nondegeneracy assumption. The next result shows that, when
these minima are finite and nondegenerate,
they are projectively repelling near this
threshold, so only their invariant lines
remain convergent.

\begin{theorem}[Terminal collapse of the convergence basin]\label{thm:terminal}
Assume $d\ge2$ and the global minima of $q_H$ on $\mathbb{RP}^{d-1}$ form a finite, nonempty set $\{[v_1],\ldots,[v_k]\}$, each with positive-definite spherical Hessian. Then there exists
\[
\eta_*\in\left(\max\left\{\frac2{S_H},\frac1{\mu_H}\right\},\frac2{\mu_H}\right)
\]
such that for every $\eta_*<\eta<2/\mu_H$,
\begin{equation}
\cB_\eta=\bigcup_{j=1}^k\operatorname{span}(H^{-1/2}v_j),
\label{eq:terminal-collapse}
\end{equation}
and every other nonzero initialization diverges exponentially. In particular, almost every initialization diverges throughout this interval, strictly below $2/\mu_H$.
\end{theorem}

\noindent The proof uses only the invariant-divergence region and the fact that a nondegenerate minimum becomes a projective repeller near $2/\mu_H$. It does not assume that typical projective orbits converge to critical rays. Theorems \ref{thm:escape} and \ref{thm:terminal} describe the two global constraints on the intermediate window; its interior can exhibit a more varied basin structure.

\section{Constructions and examples}\label{sec:examples}

Verifications and proofs for this section are in \cref{app:examples}. We first construct arbitrarily many distinct critical-direction thresholds in dimension two.

\begin{proposition}[Arbitrarily many edges]
\label{prop:many}
For every integer $m\ge2$, there exists
a norm on $\R^2$, with $\Phi$ that is $\;C^\infty$ away from $0$, has $\nabla^2\Phi \succ0$ and a positive-definite quadratic objective for which $q_H$ has at least $m$ nondegenerate strict local maxima with pairwise distinct critical values. Each yields a distinct attracting-direction radial threshold.
\end{proposition}

\noindent These thresholds are robust:
\begin{proposition}[Robustness of the edges]
\label{prop:robust}
Suppose \(q_H\) has finitely many nondegenerate strict local maxima with pairwise distinct critical values. Then every sufficiently small admissible \(C^2\) perturbation preserves nearby nondegenerate local maxima with pairwise distinct critical values. Their selection thresholds and radial stability thresholds vary continuously.
\end{proposition}

\noindent The separation between a local radial threshold and $2/S_H$ can be arbitrarily large:

\begin{proposition}[Arbitrarily large separation]
\label{prop:largegap}
For every $M>0$, there exist a smooth strongly convex norm, a positive-definite $H$, and a nondegenerate strict suboptimal local maximum $u$ such that
\begin{equation}
\frac{\eta_{\rm edge}(u)}{2/S_H}>M.
\label{eq:largeratio}
\end{equation}
In other words, there can be open stable basins for step sizes arbitrarily larger than the edge $2/S_H$.
\end{proposition}

\subsection{An exact smooth two-edge example}\label{sec:l4}
Consider the norm on $\R^2$ defined by
\begin{equation}
\|x\|_{\rm reg}
:=\left(\|x\|_4^2+\frac{1}{\sqrt{2}}\|x\|_2^2\right)^{1/2},
\qquad
H=\frac{1}{2\sqrt{2}}
\begin{pmatrix}
5&-1\\
-1&5
\end{pmatrix}.
\label{eq:exactH}
\end{equation}
The regularization ensures that the Hessian of \(\Phi\) is uniformly positive definite away from the origin, unlike that of the squared \(\ell_4\) norm, as $\Phi(x)=\tfrac12\|x\|_{\rm reg}^2$ is smooth away from
the origin and
$\nabla^2\Phi(x)\succeq 2^{-1/2}I\succ0$
for every $x\ne0$.
Thus the example satisfies the regularity assumptions.
On the projective circle \(\mathbb{RP}^1\), \(q_H\) has exactly two nondegenerate local maxima, represented by \((1,1)\) and \((1,-1)\), with critical values and radial edges
\[
\lambda_{\rm sub}=1,
\quad
\lambda_{\rm top}=S_H=\frac32 \quad \text{ and } \quad \frac2{S_H}=\frac43,
\quad
\frac2{\lambda_{\rm sub}}=2.
\]
The transverse eigenvalues are
\[
\beta_{\rm sub}=\frac34,
\qquad
\beta_{\rm top}=\frac12,
\]
so the suboptimal and global maxima become locally exponentially
attracting in projective space for $\eta>8/7$ and $\eta>1$,
respectively. Consequently, throughout
\begin{equation}
\frac43<\eta<2,
\label{eq:exactcoexist}
\end{equation}
there is an open cone of divergent initializations associated
with the global maximum and an open cone of convergent
initializations associated with the suboptimal maximum.
The calculations are given in \cref{app:l4}.
\Cref{fig:landscape} shows the landscape $q_H$, and
\cref{fig:growth} illustrates the corresponding basins.

\begin{figure}[t]
\centering
\includegraphics[width=.85\linewidth]{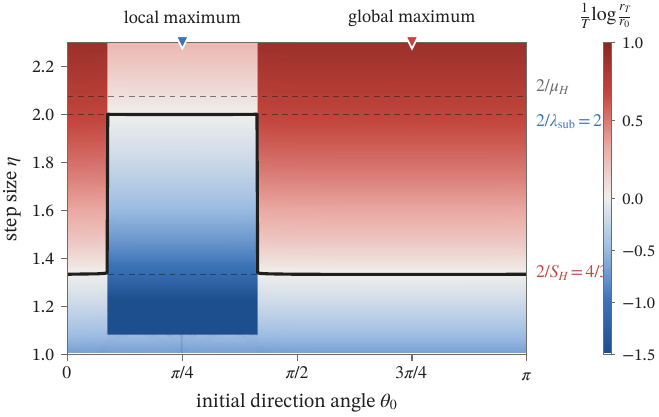}
\caption{Finite-time radial growth $T^{-1}\log(r_T/r_0)$
in the conjugate $z$-coordinates for the example
\eqref{eq:exactH}, over initial directions $\theta_0$
and step sizes $\eta$ ($T=400$; colour scale clipped
to $[-1.5,1]$). Blue denotes negative finite-time growth
and red denotes positive growth; the black contour is
the finite-time zero-growth level set.
For $4/3<\eta<2$, open sectors of convergent and
divergent initializations coexist. For these parameters,
the convergent sector occupies approximately one third
of initial directions in the displayed range.
The dashed levels mark $2/S_H$, $2/\lambda_{\rm sub}$,
and $2/\mu_H$.}
\label{fig:growth}
\end{figure}

\subsection{A three-dimensional illustration}\label{sec:3d}
In dimension three the direction dynamics live on $\mathbb{RP}^2$, and the regimes described by \cref{thm:modal,thm:escape} can be seen directly (though this example is for illustration only, as the norm is not strictly positive definite, but could be regularized to be). \Cref{fig:shells3d} shows an example in which $q_H$ has exactly two local maxima: the global maximum $u_1$, with $\lambda_1=S_H$, and a suboptimal maximum $u_2$, with $\lambda_2<S_H$. Its parameters and the computation of all thresholds are given in \cref{app:3d}. The relevant thresholds are ordered as
$
\eta_{\rm sel}(u_1)<2/{S_H}<\eta_{\rm sel}(u_2)<2/{\lambda_2}<\frac2{\mu_H},
$
and each panel of \cref{fig:shells3d} lies in a different interval between them. Below $2/S_H$ every initialization converges. Between $2/S_H$ and $\eta_{\rm sel}(u_2)$ the line through $u_2$ still converges, since $|1-\eta\lambda_2|<1$, but $[u_2]$ does not attract nearby directions, and numerically at $\eta=0.70$, every sampled initialization diverges. Above $\eta_{\rm sel}(u_2)$ the direction $[u_2]$ becomes attracting, and an open cone of convergent initializations appears around it (\cref{cor:suboptimal}). Above $2/\lambda_2$ this cone diverges as well. Throughout, the convergent cone stays outside the forward-invariant region $\mathcal U_\eta$, which grows with $\eta$.

\section{Discussion}\label{sec:discussion}

We have shown that the almost-everywhere sharp stability
threshold of Euclidean gradient descent characterizes
Hilbertian geometry. Beyond this setting, distinct
attracting critical directions can have different radial
stability thresholds, producing coexisting convergent and
divergent basins. The global analysis further constrains
their geometry: the region $\{q_H>2/\eta\}$ is
forward-invariant and divergent, while, under nondegeneracy
of the global minima, almost every initialization diverges
sufficiently close to $2/\mu_H$. These results constrain
the convergence basins as $\eta$ varies, but do not imply
that the basins themselves are nested. The results characterize stability near nondegenerate critical directions and constrain the global convergence basin, but do not classify all projective dynamics in the intermediate regime.

\subsection{On Muon}
The largest motivation for this work comes from the general interest in non-Euclidean norms, following the success of Muon \cite{jordan2024muon,bernstein2024old}. It is normally introduced as spectral norm descent, which does not satisfy our smoothness assumptions. However, considering finite even $p$-Schatten norms \cite{shumaylov2026muonspecialrandominverted}
with potentials
$
\Phi_{p}(X)
=\frac12\|X\|_{S_p}^2
$
are smooth and approach the squared spectral norm as $p\to\infty$.
The corresponding limiting dual potential is the squared nuclear norm. Although this limit is not globally smooth, it is smooth at full-rank matrices, allowing a local application of our projective stability analysis. Interestingly, in this limit the interval between the edge and radial instability collapses.

\begin{proposition}[Attracting critical directions of spectral descent are radially unstable]
Let
$
Q(Z)=\frac12\|\mathcal AZ\|_*^2,
$
where $\mathcal A=H^{1/2}$ is a positive-definite
operator on $\mathbb R^{m\times n}$, with
$\min(m,n)\ge2$. Suppose $u$ is a nondegenerate
strict local maximum of $q=2Q$ on the Frobenius unit
sphere, with $\mathcal Au$ of full rank and critical
value $\lambda$. Then the smallest transverse eigenvalue of
$\nabla^2Q(u)$ is zero, and hence
\[
\eta_{\mathrm{sel}}(u)
=\eta_{\mathrm{edge}}(u)
=\frac{2}{\lambda}.
\]
Thus $[u]$ is projectively linearly unstable
throughout its radially contracting regime
(except where the projective map is undefined),
and locally exponentially attracting for
$\eta>2/\lambda$.
\end{proposition}

The proof is given in \Cref{app:spectral}.

\bibliographystyle{plainnat}
\bibliography{bib}

\appendix
\crefalias{section}{appendix}
\crefalias{subsection}{subappendix}

\section{Proofs of the structural results}\label{app:structural}

\subsection{Radial/projective dynamics}\label{app:normal}

\begin{lemma}[Critical-point correspondence]
\label{lem:critical-duality}
The map $[x]\mapsto[Ax]$ identifies
the critical directions of $R_H$ and $q_H$,
preserving critical values and nondegenerate
strict local maxima.
\end{lemma}

\begin{proof}[Proof of \cref{lem:critical-duality}]
At non-zero $x$ we have
\begin{equation*}
    \nabla R_H(x) = \frac{1}{\Phi(x)} \left(Hx - R_H(x)\nabla \Phi(x)\right)
\end{equation*}
Hence, the critical-point equation for $R_H$ is $Hx=\lambda\nabla\Phi(x)$, with $\lambda=R_H(x)$. Assume this holds. Since $\nabla\Phi^*$ is the inverse of $\nabla\Phi$ and is homogenous of degree $1$, we obtain
\begin{equation*}
    \nabla\Phi^*(Hx) = \nabla \Phi^*(\lambda \nabla \Phi(x)) = \lambda x
\end{equation*}
Therefore,
\[
\nabla Q(Ax)=A\nabla\Phi^*(Hx)=\lambda Ax.
\]
Euler's identity then gives $q_H(Ax/\|Ax\|_2)=\lambda$. The converse follows as $A$ is invertible.

Fix a critical direction $[x]$ and choose its representative $x$ so that $\|x\|=1$. Let $G = \nabla^2 \Phi(x) \succ 0$. Then by homogeneity we have
\begin{equation*}
    Gx = \nabla\Phi(x)
\end{equation*}
Hence the critical point equation for $R_H$ gives
\begin{equation*}
    (H - \lambda G)x = 0
\end{equation*}

We compute
\begin{equation*}
    \nabla^2 R_H = 2(H - \lambda G)
\end{equation*}
This vanishes on $x$, hence its signature is equal to the signature upon restriction to the tangent space of the norm sphere at $x$. 

The corresponding critical point on the Euclidean sphere is $u = Ax / \sqrt{\lambda}$. We compute the Hessian at this point to be
\begin{equation*}
    \nabla^2 q_H(u) = 2A \nabla^2 \Phi^*(Au) A 
\end{equation*}
We compute this middle term as follows. Using that $\nabla \Phi$ and $\nabla \Phi^*$ are inverse, we obtain
\begin{equation*}
    \nabla^2 \Phi^*(\nabla \Phi(x)) = (\nabla^2 \Phi(x))^{-1} = G^{-1}
\end{equation*}
By the critical point equation
\begin{equation*}
    \nabla \Phi(x) = \frac{Hx}{\lambda} = \frac{Au}{\sqrt{\lambda}}
\end{equation*}
The Hessian of $\Phi^*$ is $0$-homogenous hence
\begin{equation*}
    \nabla^2 \Phi^*(Au) = \nabla^2 \Phi^*\left(\frac{Au}{\sqrt{\lambda}}\right) = G^{-1}
\end{equation*}
Combining this,
\begin{equation*}
    \nabla^2q_H(u) = 2AG^{-1}A
\end{equation*}
We can now compute the Hessian of $q_H$ on the unit sphere. Let $\gamma$ be a smooth curve on $S^{d-1}$ such that $\gamma(0) = u$, $\gamma'(0) = w \in u^{\perp} = T_uS^{d-1}$. Then we have
\begin{equation*}
    \frac{d^2}{ds^2} q_H(\gamma(s)) \rvert_{s=0} = w^T \nabla^2 q_H (u) w  + \langle \nabla q_H(u), \gamma''(0)\rangle = w^T \nabla^2 q_H(u) w + 2\lambda \langle u, \gamma''(0)\rangle 
\end{equation*}
Double differentation of $\gamma(s)^T\gamma(s) = 1$ gives that
\begin{equation*}
    \langle u, \gamma''(0)\rangle = - ||w||_2^2
\end{equation*}
Hence the spherical Hessian is
\begin{equation}\label{eq:SphericalHessian}
\operatorname{Hess}_{\mathbb S^{d-1}}q_H(u)
=
2\left(AG^{-1}A-\lambda I\right)
\big|_{u^\perp}.
\end{equation}
Set
\[
C=G^{-1/2}HG^{-1/2},
\qquad L=AG^{-1/2}.
\]
Then
\[
H-\lambda G
=G^{1/2}(C-\lambda I)G^{1/2},
\]
while
\[
AG^{-1}A=LL^\top,
\qquad L^\top L=C.
\]
By the polar decomposition, there exists an orthogonal
matrix $U$ such that $L=UC^{1/2}$. Consequently,
\[
AG^{-1}A-\lambda I
=U(C-\lambda I)U^\top.
\]
Thus $H-\lambda G$ and $AG^{-1}A-\lambda I$
have the same inertia. Both have a radial null
direction, represented by $x$ and $u$, respectively.
Since the tangent spaces of the corresponding norm
and Euclidean spheres are complementary to these
radial directions, their restricted Hessians have
the same inertia. In particular, either Hessian is
negative definite on its tangent space if and only
if the other is. Hence nondegenerate strict local
maxima correspond.
\end{proof}

\begin{proof}[Proof of \cref{prop:radial}]
Because $Q$ is $2$-homogeneous,
\[
Q(ru)=r^2Q(u)=\frac{r^2}{2}q_H(u).
\]
Differentiation in Euclidean polar coordinates gives \eqref{eq:polargrad}. Substituting into \eqref{eq:qgd} yields
\[
z_{t+1}
=r_t\left[(1-\eta q_H(u_t))u_t-\frac\eta2\nabla_Sq_H(u_t)\right].
\]
The two terms are orthogonal, which gives \eqref{eq:rho}; division by the norm gives \eqref{eq:projmap}.
\end{proof}

\begin{proof}[Proof of \cref{prop:directional}]
Let $g=Au$ and $d=\nabla\Phi^*(g)$. Then
\[
\nabla Q(u)=Ad,
\qquad
q_H(u)=\|g\|_*^2.
\]
The duality-map identities imply $\|d\|=\|g\|_*$, hence
\[
D_H(u)
=\frac{d^\top Hd}{\|d\|^2},
\]
which is exactly the quadratic curvature along the update direction and is bounded above by $S_H$. Using \eqref{eq:polargrad},
\[
\|\nabla Q(u)\|_2^2
=q_H(u)^2+\frac14\|\nabla_Sq_H(u)\|_2^2,
\]
which proves \eqref{eq:Ddecomp}. Substitution into \eqref{eq:rho} gives \eqref{eq:rhoD}.
\end{proof}

\section{Proofs of the local stability results}\label{app:modal}

\begin{proof}[Proof of \cref{thm:modal}]
At a critical point $u$ of $q_H$, \eqref{eq:polargrad} gives
\[
\nabla Q(u)=\lambda u.
\]
Thus the ray is invariant and its radial multiplier is $1-\eta\lambda$.

By Euler's identity for a $2$-homogeneous function,
\[
B_uu=\nabla Q(u)=\lambda u.
\]
Let $h\in u^\perp$ satisfy $B_uh=\beta_i h$. Linearizing $G_\eta(z)=z-\eta\nabla Q(z)$ gives
\[
DG_\eta(u)h=(1-\eta\beta_i)h.
\]
Relative to the radial multiplier $1-\eta\lambda$, the projective multiplier is therefore \eqref{eq:kappa}.

At a critical point, by \eqref{eq:SphericalHessian} the spherical Hessian is
\[
\nabla_S^2q_H(u)
=2(B_u-\lambda I)|_{u^\perp}.
\]
Thus the spherical Hessian is negative definite precisely when $\beta_i<\lambda$ for all $i$, which is equivalent to $u$ being a nondegenerate strict local maximum.

At such a maximum, strict linear contraction in projective coordinates is equivalent to
\[
|1-\eta\beta_i|<|1-\eta\lambda|
\quad\text{for all }i.
\]
Squaring and factoring gives
\[
\eta(\lambda-\beta_i)\bigl(\eta(\lambda+\beta_i)-2\bigr)>0.
\]
Since $\lambda>\beta_i$, this is equivalent to
\[
\eta>\frac2{\lambda+\beta_i}.
\]
Maximizing the bound over $i$ gives \eqref{eq:selection}. In this range, the fixed point is locally exponentially attracting by the contraction theorem for $C^1$ maps. Below the threshold (away from $\eta=1/\lambda$), some transverse multiplier has modulus greater than one. 

Finally, if a nonzero projective orbit converges to $[u]$ and
$\eta\ne1/\lambda$, continuity of \eqref{eq:rho} gives
\[
\rho_\eta(u_t)\longrightarrow |1-\eta\lambda|>0.
\]
Since
\[
r_t=r_0\prod_{k=0}^{t-1}\rho_\eta(u_k),
\]
we have
\begin{equation*}
    \frac{1}{t}\log r_t
    = \frac{1}{t}\left(
        \log r_0 + \sum_{k=0}^{t-1}\log \rho_\eta(u_k)
    \right).
\end{equation*}
Moreover,
$\log \rho_\eta(u_k) \to \log \lvert 1-\eta\lambda\rvert$.
Thus, for any fixed $\varepsilon>0$, there exists $K$ such that
for all $k>K$,
\begin{equation*}
    \left|
        \log \rho_\eta(u_k)
        - \log \lvert 1-\eta\lambda\rvert
    \right| < \varepsilon.
\end{equation*}
For $t>K+1$, we therefore have
\begin{align*}
    \left|
        \frac{1}{t}\log r_t
        - \log \lvert 1-\eta\lambda\rvert
    \right|
    &=
    \left|
        \frac{C}{t}
        + \frac{1}{t}\sum_{k=K+1}^{t-1}
        \left(
            \log \rho_\eta(u_k)
            - \log \lvert 1-\eta\lambda\rvert
        \right)
    \right| \\
    &\le
    \frac{\lvert C\rvert+(t-1-K)\varepsilon}{t},
\end{align*}
where
\begin{equation*}
    C = \log r_0
        + \sum_{k=0}^{K}
        \left(
            \log \rho_\eta(u_k)
            - \log \lvert 1-\eta\lambda\rvert
        \right)
\end{equation*}
is a constant independent of $t$.
Therefore, for all sufficiently large $t$,
\begin{equation*}
    \left|\frac{1}{t} \log r_t - \log |1 -\eta\lambda|\right| < 2\varepsilon.
\end{equation*}
Since $\varepsilon>0$ was arbitrary,
\begin{equation*}
    \frac{1}{t} \log r_t \rightarrow \log |1 - \eta \lambda|.
\end{equation*}
Because $L(x_t)=r_t^2/2$, this proves \eqref{eq:modegrowth}.
\end{proof}

\begin{proof}[Proof of \cref{cor:first}]
At a nondegenerate global maximum, $\lambda=S_H$ and $0< \beta_{\min}<S_H$, so
\[
\eta_{\rm sel}<\frac2{S_H}.
\]
For every $\eta>2/S_H$, $[u_\star]$ is therefore projectively attracting and its radial multiplier has magnitude $|1-\eta S_H|>1$. Its local projective basin is open, hence its cone is an open divergent set.
\end{proof}

\begin{proof}[Proof of \cref{cor:suboptimal}]
The interval \eqref{eq:stablebeyond} is nonempty because both $2/S_H$ and $\eta_{\rm sel}(u)$ are strictly smaller than $2/\lambda$. In this interval $[u]$ is therefore projectively attracting while $|1-\eta\lambda|<1$, so its open local projective basin lifts to an open cone of exponentially convergent trajectories.
\end{proof}

\section{Proofs for the global stability regions}\label{app:global}

\begin{proof}[Proof of \cref{thm:escape}]
Write $q=q_H(u)$, $g=\nabla Q(u)$, $D=\|g\|_2^2/q$ and $\rho=\|u-\eta g\|_2$. By Euler's identity $g^\top u=q$, and Cauchy--Schwarz gives $D\ge q$. Put $a=\eta q$ and $b=\eta D$. The radial identity yields
\[
\rho^2=1-2a+ab,
\qquad
(b-1)^2-\rho^2=(b-a)(b-2).
\]
If $[u]\in\mathcal U_\eta$, then $b\ge a>2$. Therefore $\rho\ge a-1>1$ by \eqref{eq:rho}, and $b-1\ge\rho>0$. In particular, the next projective iterate exists. Because $Q$ is even, the next projective direction can be represented by $v=(\eta g-u)/\rho$. The supporting-hyperplane inequality for the convex function $Q$ gives
\begin{align*}
q_H(P_\eta([u]))=2Q(v)
&\ge 2Q(u)+2g^\top(v-u)\\
&=q\left(\frac{2(b-1)}{\rho}-1\right)\ge q.
\end{align*}
This proves invariance. For any such initial direction, $q_H(u_t)\ge q_H(u_0)>2/\eta$, so
\[
r_t\ge r_0\bigl(\eta q_H(u_0)-1\bigr)^t\longrightarrow\infty
\]
exponentially.

Write $U_n \subseteq \mathbb{RP}^{d-1}$ for the locus where $P_{\eta}^n$ is defined. This is open, as it is precisely the locus where the $n$-th iterate of the gradient descent step does not vanish. Then $P^n_\eta$ is a continuous map $U_n \rightarrow \mathbb{RP}^{d-1}$. We have that $(P^n_\eta)^{-1}(\mathcal{U}_\eta)$ is the locus of directions whose projective gradient descent iterations are defined to step $n$ and fall into $\mathcal{U}_\eta$, this is open in $U_n$ and hence open in $\mathbb{RP}^{d-1}$. Note that
\begin{equation*}
    \mathcal{E}_\eta = \bigcup_{n \geq 0} (P^n_\eta)^{-1}(\mathcal{U}_\eta)
\end{equation*}

Thus $\mathcal E_\eta$, being the union of the open preimages of $\mathcal U_\eta$, is open. Every trajectory in $\mathcal{E}_\eta$ is eventually exponentially diverging by the argument above. Therefore no convergent trajectory can enter $\mathcal U_\eta$, and \eqref{eq:stable-envelope} then follows.
\end{proof}

\begin{proof}[Proof of \cref{thm:terminal}]
Let $\eta_0=2/\mu_H$. At a global minimum $[v_j]$, the spherical Hessian is positive definite, so by \eqref{eq:SphericalHessian} its transverse eigenvalues satisfy $\beta_{j,\ell}>\mu_H$. At $\eta_0$ the denominator $1-\eta_0\mu_H=-1$, and the projective derivative has multipliers
\[
\kappa_{j,\ell}(\eta_0)=\frac{2\beta_{j,\ell}}{\mu_H}-1>1.
\]
The derivative on the tangent space is symmetric in orthonormal projective coordinates. Hence its smallest singular value is greater than one.

We now make the repulsion global away from the critical rays. Equip $\mathbb{RP}^{d-1}$ with its standard Riemannian metric and choose disjoint coordinate balls $V_j$ centered at the finitely many $[v_j]$, whose closures are disjoint. Since $P_{\eta_0}([v_j])=[v_j]$ and its derivative has smallest singular value greater than one, continuity of the derivative in both $\eta$ and direction gives $c>1$, $\delta>0$ and $\varepsilon>0$ such that, in normal coordinates $\xi_j$ centered at $[v_j]$, whenever $\eta_0-\varepsilon<\eta\le\eta_0$ and $\operatorname{dist}([u],[v_j])<\delta$,
\begin{equation}
\operatorname{dist}(P_\eta([u]),[v_j])
\ge c\,\operatorname{dist}([u],[v_j]),
\label{eq:local-repel}
\end{equation}
and $P_\eta([u])\in V_j$. Indeed, after shrinking $\delta$ and $\varepsilon$, the derivative of the coordinate expression differs in operator norm from its derivative at $(\eta_0,[v_j])$ by less than half the latter's smallest singular value minus one. The mean-value integral along the radial coordinate segment then proves \eqref{eq:local-repel}. Shrinking $\delta$ and $\varepsilon$ again we can ensure $P_{\eta}([u]) \in V_j$ by continuity of $P_{\eta}$ and openness of $V_j$. We may also require that $\overline{B_\delta([v_j])}\subset V_j$ and that $V_j$ meets no $B_\delta([v_\ell])$ for $\ell\ne j$.

Because the minimum set is finite, compactness of projective space implies that, for $\eta$ sufficiently close to $\eta_0$ from below,
\begin{equation}
K_\eta:=\{[u]:q_H(u)\le2/\eta\}
\subset\bigcup_{j=1}^k B_\delta([v_j]).
\label{eq:small-sublevel}
\end{equation}
Indeed, let $V_\delta = \mathbb{RP}^{d-1} \setminus \bigcup_{j=1}^k B_{\delta}([v_j])$, this is a closed subset of $\mathbb{RP}^{d-1}$. Hence it is compact. The collection of opens $K_\eta^c \cap V_\delta$ for $\eta < \eta_0$ form an open cover of $V_\delta$. Further, they are obviously filtered in the sense that $\eta < \eta'$ implies $K_{\eta}^c \subseteq K_{\eta'}^c$. Therefore compactness implies there is some $\eta' < \eta_0$ such that $V_\delta \subseteq K_{\eta'}^c$. Any $\eta$ such that $\eta' \leq \eta < \eta_0$ will then satisfy \eqref{eq:small-sublevel}.

Choose $\eta_*$ sufficiently close to $\eta_0$ from below such that \eqref{eq:local-repel} and \eqref{eq:small-sublevel} hold on $(\eta_*,\eta_0)$ and further that $\eta_*>\max\{2/S_H,1/\mu_H\}$. This second condition is possible as $S_H > \mu_H$ as $q_H$ cannot be constant. Indeed, a constant function on a positive dimensional smooth manifold cannot have isolated critical points.

For $\eta>1/\mu_H$ the projective update is actually globally defined as if $u-\eta\nabla Q(u)=0$, then Euler's identity gives $q_H(u)=1/\eta<\mu_H$, a contradiction. Consider now a direction whose entire orbit remains in $K_\eta$. By \eqref{eq:small-sublevel} it belongs at each time to one of the balls $B_\delta([v_j])$. The image containment in $V_j$ and the separation of the balls show that it cannot move between different balls while remaining in $K_\eta$. Thus it stays in a single ball. If its initial direction differs from $[v_j]$, repeated application of \eqref{eq:local-repel} makes its distance from $[v_j]$ at least $c^t$ times the initial distance, contradicting confinement to that ball. Therefore every direction other than the $[v_j]$ eventually leaves $K_\eta$, enters $\mathcal U_\eta$, and diverges exponentially by \cref{thm:escape}.

Finally, each minimum ray is invariant and satisfies $z_{t+1}=(1-\eta\mu_H)z_t$. Since $\eta\in(1/\mu_H,2/\mu_H)$, its iterates converge to zero. This proves \eqref{eq:terminal-collapse} and the almost-everywhere claim, since finitely many lines have Lebesgue measure zero.
\end{proof}

\section{Proof of the Hilbertian characterization}\label{app:hilbert}

\begin{lemma}
\label{lem:parallelogram}
Let $G_v:=\nabla^2\Phi(v)$. If
\begin{equation}
 d^\top G_vd\le\|d\|^2
 \qquad\text{for all }v\neq0\text{ and }d\in\R^d,
 \label{eq:hessianbound}
\end{equation}
then $\|\cdot\|$ satisfies the parallelogram law and is Hilbertian.
\end{lemma}

\begin{proof}
Fix $x,y$ and set $f(t)=\Phi(x+ty)$. Wherever $x+ty\neq0$,
\[
f''(t)=y^\top G_{x+ty}y\le\|y\|^2.
\]
The function \(f(t)=\Phi(x+ty)\) is continuously differentiable on \(\mathbb R\). Except possibly at the unique value for which \(x+ty=0\), its second derivative exists and satisfies \(f''(t)\leq\|y\|^2\). Splitting the integral at that point, when it exists, and using continuity of \(f'\) yields the triangular-kernel inequality.
\[
\|x+y\|^2+\|x-y\|^2\le2\|x\|^2+2\|y\|^2.
\]
Apply the same inequality to $(x+y)/2$ and $(x-y)/2$ to obtain the reverse inequality. Hence equality holds, and the Jordan--von Neumann theorem \citep{jordan1935inner} implies that the norm is induced by an inner product.
\end{proof}

\begin{proof}[Proof of \cref{thm:hilbert}]
Suppose first that $\|x\|^2=x^\top Px$ with $P\succ0$. Then
\[
R_H(x)=\frac{x^\top Hx}{x^\top Px}
\]
is an ordinary generalized Rayleigh quotient. After $y=P^{1/2}x$, its local maxima are precisely the top eigenspace of $P^{-1/2}HP^{-1/2}$, proving (i)$\Rightarrow$(ii). The update becomes
\[
y_{t+1}=(I-\eta M)y_t,
\qquad
M=P^{-1/2}HP^{-1/2}.
\]
If $\eta>2/\lambda_{\max}(M)=2/S_H$, every initialization with a nonzero top-eigenspace component diverges; the remaining set is a proper linear subspace. Thus (i)$\Rightarrow$(iii).

Now assume the norm is not Hilbertian. By \cref{lem:parallelogram}, \eqref{eq:hessianbound} fails. Hence there is a unit vector $v$ and a vector $d$ such that
\[
d^\top G_vd>\|d\|^2,
\qquad
G_v:=\nabla^2\Phi(v).
\]
Equivalently $S_{G_v}>1$. Let
\[
g:=\nabla\Phi(v)=G_vv
\]
and, for $0<\varepsilon<1$, define
\begin{equation}
H_\varepsilon:=(1-\varepsilon)G_v+\varepsilon gg^\top.
\label{eq:Heps}
\end{equation}
Since $g^\top v=1$,
\[
H_\varepsilon v=g,
\]
so $v$ is a critical point of $R_{H_\varepsilon}$ with value $1$. On the tangent space
\[
T_v=\{z:g^\top z=0\},
\]
we have
\[
H_\varepsilon z=(1-\varepsilon)G_vz.
\]
Therefore the constrained second variation at $v$ is
\[
2z^\top(H_\varepsilon-G_v)z
=-2\varepsilon z^\top G_vz<0.
\]
Thus $v$ is a strict local maximum. Since $H_\varepsilon\to G_v$ and $S_{G_v}>1$, for all sufficiently small $\varepsilon$ we still have $S_{H_\varepsilon}>1$, so this local maximum is suboptimal. This proves failure of (ii). By \cref{lem:critical-duality}, it is also a nondegenerate suboptimal maximum of $q_{H_\varepsilon}$. Thus \cref{cor:suboptimal} then gives an open convergent cone for some $\eta>2/S_{H_\varepsilon}$, proving failure of (iii).
\end{proof}

\section{Proofs and calculations for the examples}\label{app:examples}

\subsection{The two-edge example}\label{app:l4}

\begin{proof}[Verification of the claims in \cref{sec:l4}]
\emph{Regularity.} The function $\|x\|_4^2=(x_1^4+x_2^4)^{1/2}$ is convex and smooth away from the origin, so $\nabla^2\bigl(\tfrac12\|x\|_4^2\bigr)\succeq0$ there. Adding $\tfrac{1}{2\sqrt2}\|x\|_2^2$ gives $\nabla^2\Phi(x)\succeq2^{-1/2}I$ for all $x\neq0$.

\emph{Critical values and $S_H$.} We work with the generalized Rayleigh quotient $R_H(x)=x^\top Hx/\|x\|_{\rm reg}^2$, whose critical values agree with those of $q_H$. By homogeneity write $x=(1,t)$. Then
\begin{equation}
R_H(1,t)
=\frac{5(1+t^2)-2t}{2\sqrt2\sqrt{1+t^4}+2(1+t^2)}.
\label{eq:Rexact}
\end{equation}
Since $H$ and the norm are invariant under swapping coordinates, $R_H(1,t)=R_H(1,1/t)$, so the direction $(0,1)$ behaves like $t=0$. The numerator of $\partial_tR_H(1,t)$ vanishes exactly when
\[
(t-1)(t+1)\Bigl(\sqrt2\,(t^2-5t+1)+\sqrt{1+t^4}\Bigr)=0.
\]
At $t=\pm1$,
\[
R_H(1,1)=1,
\qquad
R_H(1,-1)=\frac32,
\qquad
\partial_t^2R_H(1,1)=-\frac14,
\qquad
\partial_t^2R_H(1,-1)=-1,
\]
so both are nondegenerate strict local maxima. The last factor can only vanish for $t>0$; with $s=t+1/t\ge2$ it reads $\sqrt{s^2-2}=\sqrt2\,(5-s)$, whose admissible root is $s=10-4\sqrt3$. This gives the two remaining critical points, both with value
\[
R_H=\frac{5s-2}{20-2s}=2\sqrt3-\frac52,
\]
which is therefore $\mu_H$. Finally, $R_H(1,t)\le3/2$ is equivalent to $2(t^2-t+1)\le3\sqrt2\sqrt{1+t^4}$, which holds because
\[
18(1+t^4)-4(t^2-t+1)^2=2(t+1)^2(7t^2-10t+7)\ge0.
\]
Hence $S_H=3/2$.

\emph{Transverse eigenvalues.} At a critical direction $u$ of $q_H$, with $g=Au$ and $y=\nabla\Phi^*(g)$,
\[
B_u=\nabla^2Q(u)=A\,\nabla^2\Phi^*(g)\,A=A\,\bigl(\nabla^2\Phi(y)\bigr)^{-1}A.
\]
The vectors $(1,1)$ and $(1,-1)$ are eigenvectors of $H$, with eigenvalues $\sqrt2$ and $3/\sqrt2$, and $\nabla\Phi$ maps each of these lines to itself, so $y$ is parallel to $u$. Since $\nabla^2\Phi$ is $0$-homogeneous, it suffices to evaluate it at $y=(1,\pm1)$, where
\[
\nabla^2\Phi(1,\pm1)=2\sqrt2\,I-\frac{1}{\sqrt2}\,yy^\top.
\]
This has eigenvalue $\sqrt2$ along $y$ and $2\sqrt2$ orthogonal to it. Multiplying the eigenvalues of $H$ by those of $(\nabla^2\Phi)^{-1}$ along the same directions gives, at $u\parallel(1,1)$,
\[
\lambda_{\rm sub}=\sqrt2\cdot\frac1{\sqrt2}=1,
\qquad
\beta_{\rm sub}=\frac3{\sqrt2}\cdot\frac1{2\sqrt2}=\frac34,
\]
and at $u\parallel(1,-1)$,
\[
\lambda_{\rm top}=\frac3{\sqrt2}\cdot\frac1{\sqrt2}=\frac32,
\qquad
\beta_{\rm top}=\sqrt2\cdot\frac1{2\sqrt2}=\frac12.
\]
\cref{thm:modal} therefore gives
\[
\eta_{\rm sel,sub}=\frac2{1+3/4}=\frac87,
\qquad
\eta_{\rm edge,sub}=2,
\qquad
\eta_{\rm sel,top}=\frac2{3/2+1/2}=1,
\qquad
\eta_{\rm edge,top}=\frac43.
\]
Since $\max\{2/S_H,\eta_{\rm sel,sub}\}=4/3$, coexistence holds throughout \eqref{eq:exactcoexist}.
\end{proof}

\subsection{Arbitrarily many edges}\label{app:many}

\begin{proof}[Proof of \cref{prop:many}]

Take $H=I$. Fix $m\ge 2$ and let
$\theta_j=j\pi/m$, $j=0,\ldots,m-1$. Choose smooth
$\pi$-periodic functions $b_j$ with pairwise disjoint
supports modulo $\pi$, satisfying
\[
b_j(\theta_j)=1,\qquad b_j'(\theta_j)=0.
\]
For example, one may take bump functions on a small enough neighborhood around $\theta_j$ and then extend periodically. 

For pairwise distinct $a_j\in\mathbb R$, set
\[
f(\theta)=\cos(2m\theta)
+\delta\sum_{j=0}^{m-1}a_jb_j(\theta).
\]
At each $\theta_j$,
\[
f'(\theta_j)=0,\qquad
f''(\theta_j)=-4m^2+\delta a_jb_j''(\theta_j)<0
\]
for sufficiently small $\delta>0$. Thus the
$\theta_j$ are nondegenerate strict local maxima
with pairwise distinct values
$f(\theta_j)=1+\delta a_j$.

\noindent For $z=r(\cos\theta,\sin\theta)\ne0$ and $\varepsilon>0$, define
\[
Q_\varepsilon(z)=\frac{r^2}{2}(1+\varepsilon f(\theta))^2,
\qquad Q_\varepsilon(0)=0.
\]
Because $f$ is smooth and $\pi$-periodic, $Q_\varepsilon$ is even, $C^\infty$ away from zero, and positively $2$-homogeneous. Its Hessian depends continuously on $(\varepsilon,\theta)$ and equals $I$ when $\varepsilon=0$. By compactness of the unit circle, for sufficiently small positive $\varepsilon$,
\[
\nabla^2Q_\varepsilon(z)\succeq\tfrac12 I
\qquad (z\ne0),
\]
and $1+\varepsilon f(\theta)>0$ for all $\theta$. Thus $\|z\|_*:=\sqrt{2Q_\varepsilon(z)}$ is a norm: its unit ball is the centrally symmetric convex set $\{Q_\varepsilon\le1/2\}$, and the gauge equals $\sqrt{2Q_\varepsilon}$ by homogeneity. The Legendre dual potential $\Phi=Q_\varepsilon^*$ is $C^\infty$ off zero with
\[
\nabla^2\Phi(x)=\bigl(\nabla^2Q_\varepsilon(\nabla\Phi(x))\bigr)^{-1}\succ0
\qquad(x\ne0),
\]
so its associated primal norm satisfies all standing assumptions.

Finally, $q_I(u_\theta)=2Q_\varepsilon(u_\theta)=(1+\varepsilon f(\theta))^2$. The selected maxima of $f$ therefore give nondegenerate strict local maxima of $q_I$ with pairwise distinct critical values and hence distinct attracting-direction radial thresholds.
\end{proof}

\subsection{Robustness}\label{app:robust}

\begin{proof}[Proof of \cref{prop:robust}]
Write $\Phi=\frac12\|\cdot\|^2$, $A=H^{1/2}$,
and $Q(z)=\Phi^*(Az)$. We first verify that
sufficiently small admissible $C^2$ perturbations
of $(\Phi,H)$ induce small $C^2$ perturbations
of $q_H=2Q|_{\mathbb S^{d-1}}$.

Indeed, since $\nabla^2\Phi\succ0$ away from zero,
$\nabla\Phi$ is a $C^1$ diffeomorphism of
$\mathbb R^d\setminus\{0\}$, with inverse
$\nabla\Phi^*$. The inverse-function theorem,
applied uniformly on compact sets away from zero,
shows that this inverse varies continuously in
$C^1$ under the prescribed perturbations.
Moreover,
\[
\nabla Q(z)=A\nabla\Phi^*(Az),
\]
and
\[
\nabla^2Q(z)
=A\bigl[\nabla^2\Phi(\nabla\Phi^*(Az))\bigr]^{-1}A.
\]
Since $H\mapsto H^{1/2}$ is analytic on the
positive-definite cone, these identities imply
$C^2$ continuity of $Q$, and hence of $q_H$,
on the Euclidean unit sphere.

Let $[u_1],\ldots,[u_m]$ be the specified
nondegenerate strict local maxima, with pairwise
distinct values $\lambda_1,\ldots,\lambda_m$.
Choose disjoint projective coordinate
neighbourhoods of these points. In each chart,
the critical-point equation is
\[
\nabla q_H=0.
\]
Its derivative with respect to the coordinates
is the spherical Hessian, which is negative
definite and therefore invertible at each $u_j$.
The implicit-function theorem yields a unique
nearby critical point $[\widetilde u_j]$ for
every sufficiently small perturbation.
Continuity of the Hessian ensures that these
critical points remain nondegenerate strict
local maxima.

Their critical values $\widetilde\lambda_j$
depend continuously on the perturbation.
Since the original values are pairwise distinct
and their number is finite, the perturbed
values remain pairwise distinct.

At each continued critical direction, the
transverse eigenvalues of $\nabla^2 Q$ vary
continuously, as do the critical value and
therefore
\[
\eta_{\mathrm{sel}}
=\frac{2}{\lambda+\beta_{\min}},
\qquad
\eta_{\mathrm{edge}}=\frac{2}{\lambda}.
\]
Furthermore, $S_H=\max q_H$ varies continuously
by compactness of the sphere.

Finally, every strictly suboptimal maximum
satisfies
\[
\max\left\{\frac{2}{S_H},
\eta_{\mathrm{sel}}\right\}
<\frac{2}{\lambda}.
\]
These strict inequalities persist under
sufficiently small perturbations. Thus the
corresponding open stable-beyond-$2/S_H$
intervals persist, with continuously varying
endpoints.
\end{proof}

\subsection{Arbitrarily large separation}\label{app:largegap}

\begin{proof}[Proof of \cref{prop:largegap}]
Let $p>2$ be an even integer and define the norm
\[
\|x\|=\bigl(\|x\|_p^2+\|x\|_2^2\bigr)^{1/2}
\]
on $\R^2$. Its squared-norm potential
\[
\Phi(x)=\frac12\|x\|_p^2+\frac12\|x\|_2^2
\]
is $C^\infty$ away from the origin and satisfies
\[
\nabla^2\Phi(x)
=\nabla^2\!\left(\frac12\|x\|_p^2\right)+I
\succeq I\succ0
\qquad (x\ne0).
\]
Thus the norm satisfies the standing assumptions.

Set
\[
a=2^{-1/p},\qquad
c_p=2^{1-2/p},
\qquad
v=\frac{a(1,1)}{\sqrt{1+c_p}},
\qquad
w=\frac{a(1,-1)}{\sqrt{1+c_p}}.
\]
Then $\|v\|=\|w\|=1$. Let
\[
G=\nabla^2\Phi(v),
\qquad
g=\nabla\Phi(v).
\]
By $2$-homogeneity, $Gv=g$ and $g^\top v=1$.
Direct differentiation gives
\begin{equation}
g^\top w=0,
\qquad
w^\top Gw=\frac{p-1+c_p}{1+c_p}.
\label{eq:lpcurvature}
\end{equation}

Define
\[
H:=\frac12G+\frac12gg^\top.
\]
Since $G\succ0$, we have $H\succ0$. Moreover,
\[
Hv=\frac12Gv+\frac12g(g^\top v)=g.
\]
Hence $v$ is a critical point of the generalized
Rayleigh quotient
\[
R_H(x)=\frac{x^\top Hx}{\|x\|^2}
\]
with critical value $\lambda=1$.

On the tangent space
$T_v=\{z:g^\top z=0\}$, the constrained
second variation is
\[
2z^\top(H-G)z=-z^\top Gz<0
\qquad (z\in T_v\setminus\{0\}).
\]
Thus $v$ is a nondegenerate strict local maximum.

To identify its transverse eigenvalue, let
$A=H^{1/2}$ and $Q(z)=\Phi^*(Az)$.
At the corresponding critical direction,
the Hessian of $Q$ is $AG^{-1}A$, whose eigenvalues
equal those of
\[
G^{-1/2}HG^{-1/2}
=\frac12I+
\frac12(G^{-1/2}g)(G^{-1/2}g)^\top.
\]
Since
\[
g^\top G^{-1}g=g^\top v=1,
\]
these eigenvalues are $1$ and $1/2$.
The radial eigenvalue is $\lambda=1$, so the
transverse eigenvalue is $\beta=1/2$.
Consequently, the associated projective direction
is attracting for $\eta>4/3$, and its radial
stability edge is
\[
\eta_{\rm edge}=\frac2\lambda=2.
\]

On the other hand, evaluating generalized
sharpness at $w$ and using \eqref{eq:lpcurvature}
gives
\[
S_H\ge w^\top Hw
=\frac12w^\top Gw
=\frac{p-1+c_p}{2(1+c_p)}.
\]
Therefore
\[
\frac{\eta_{\rm edge}}{2/S_H}
=S_H
\ge\frac{p-1+c_p}{2(1+c_p)}
>\frac{p-1}{6},
\]
where the last inequality uses $c_p<2$.
Choosing an even integer $p>\max\{6M+1,7\}$ yields
\[
\frac{\eta_{\rm edge}}{2/S_H}>M.
\]

Since $\beta<\lambda$, the critical direction
has an open projectively attracting basin
throughout
\[
\frac43<\eta<2.
\]
As $S_H>M$, this basin remains radially stable
for step sizes arbitrarily large relative
to $2/S_H$. This proves the claim.
\end{proof}

\subsection{The three-dimensional example}\label{app:3d}
In the conjugate coordinates of \cref{sec:normal}, the example of \cref{sec:3d} is
\[
Q(z)=\frac12\|Az\|_{4/3}^2,
\qquad
A=H^{1/2}=
\begin{pmatrix}
1.153&0.174&0.020\\
0.174&1.323&0.171\\
0.020&0.171&0.982
\end{pmatrix},
\]
so the primal norm is $\|\cdot\|_4$ on $\R^3$ and $H=A^2$. This norm does not satisfy the standing assumptions everywhere: $\nabla^2\Phi$ degenerates on the coordinate axes, and correspondingly $q_H$ fails to be $C^2$ where a coordinate of $Au$ vanishes. This does not affect the results we use. At the two maxima below, every coordinate of $Au$ is nonzero, so $Q$ is smooth near them and \cref{thm:modal} applies locally. \Cref{thm:escape} uses only that $Q$ is convex, even, $2$-homogeneous and $C^1$.

The critical directions were computed numerically. Riemannian gradient ascent of $q_H$ from $2\times10^4$ uniformly random initial directions converges, up to sign, to exactly two local maxima. Their transverse eigenvalues $\beta_1,\beta_2$ were obtained from a finite-difference Hessian of $Q$ restricted to $u^\perp$:
\begin{center}
\begin{tabular}{lccccc}
\toprule
 & $u$ & $\lambda=q_H(u)$ & $\beta_1,\beta_2$ & $\eta_{\rm sel}(u)$ & $2/\lambda$\\
\midrule
$u_1$ & $(0.5381,\ 0.7226,\ 0.4339)$ & $3.5287$ & $0.5069,\ 0.7411$ & $0.4956$ & $0.5668$\\
$u_2$ & $(-0.5953,\ 0.7370,\ -0.3200)$ & $1.5636$ & $0.8395,\ 1.3350$ & $0.8323$ & $1.2791$\\
\bottomrule
\end{tabular}
\end{center}
Both are nondegenerate strict local maxima, since $\beta_i<\lambda$, and $u_1$ is the global maximum, so $S_H=3.5287$. The global minimum $\mu_H\approx0.9050$, so $2/\mu_H\approx2.210$, was computed by derivative-free minimization; it is attained where the first coordinate of $Au$ vanishes. The step sizes $\eta=0.52,\,0.70,\,1.00,\,1.35$ in \cref{fig:shells3d} therefore lie in the four intervals determined by these thresholds.

In \cref{fig:shells3d}, the colour at each direction $u$ is the sign of $T^{-1}\log(r_T/r_0)$ for the iteration \eqref{eq:qgd} started at $z_0=u$, with $T=300$; the iterate is renormalized after every step, which is exact because $\nabla Q$ is $1$-homogeneous. Trajectories show every second iterate of the projective map, joined by shortest arcs in $\mathbb{RP}^2$. Arrows show only the direction of the two-step displacement, because a single step can move a direction by tens of degrees.
\section{Muon example}\label{app:spectral}

\paragraph{The spectral-norm limit.}
Muon motivates considering matrix norms approaching the spectral
norm, but its normalized polar update is not the fixed-step
duality-map iteration studied here. For a matrix parameter
$X\in\mathbb R^{m\times n}$, put
\[
 \Phi_{p}(X)
 =\frac12\|X\|_{S_p}^{2}
  ,
 \qquad p\in\{4,6,8,\ldots\}.
\]
These potentials satisfy our smoothness and positive-definite
Hessian assumptions away from zero. As $p\to\infty$, the primal potential approaches
$\frac12\|X\|_{\mathrm{op}}^{2}$, whose convex conjugate is
$\frac12\|G\|_*^{2}$.

The limiting squared nuclear norm is not globally $C^2$, picking up singularities along rank-deficient matrices. The following local result is therefore proved directly
in the spectral limit, at points where the transformed matrix
has full rank. It also admits a precise continuation from
the unregularized finite-$p$ dual potentials.

We consider $\mathbb R^{m\times n}$ with the Frobenius inner product,
\begin{equation*}
    \langle A, B \rangle = \operatorname{Tr}(A^TB)
\end{equation*}
This induces a norm we write as $\|\cdot\|_F$.

Let $\mathcal A=H^{1/2}$
be a positive-definite, self-adjoint linear operator on this
matrix space, and define
\[
 Q_\infty(Z):=\frac12\|\mathcal AZ\|_*^2,
 \qquad
 q_\infty(u):=2Q_\infty(u),\quad \|u\|_F=1.
\]
Here $\mathcal AZ$ denotes the action of the
Hessian-square-root linear operator on the matrix
parameter, not ordinary left matrix multiplication.

\begin{proposition}
[Collapse of the selection interval in spectral geometry]
\label{prop:spectral_selection_collapse}
Let $r=\min\{m,n\}\ge2$, and let
$u\in\mathbb R^{m\times n}$ satisfy $\|u\|_F=1$.
Assume that $G:=\mathcal Au$ has rank $r$ and that
$u$ is a critical point of $q_\infty$ on the Frobenius
unit sphere. Set $\lambda=q_\infty(u)>0$ and let
\[
 B:=\left.\nabla^2Q_\infty(u)\right|_{u^\perp}
\]
denote the transverse Hessian, where $u^\perp$ is
taken with respect to the Frobenius inner product.
Then:

\begin{enumerate}
\item[(i)]
$B$ is positive semidefinite and has a kernel of
dimension at least $r-1$. In particular, its
smallest eigenvalue is
\[
 \beta_{\min}=0.
\]

\item[(ii)]
For $\eta>0$, $\eta\ne1/\lambda$, the projective map
induced by
\[
 Z\mapsto Z-\eta\nabla Q_\infty(Z)
\]
is well-defined near $[u]$ and has a transverse
multiplier
\[
 \kappa_0(\eta)=\frac{1}{1-\eta\lambda}
\]
with multiplicity at least $r-1$.

Consequently, $[u]$ is projectively linearly unstable
for every
\[
 0<\eta<2/\lambda,\qquad \eta\ne1/\lambda,
\]
though its invariant line is radially contracting
for those step sizes. 

\item[(iii)]
If, in addition, $u$ is a nondegenerate strict
local maximum of $q_\infty$ on the Frobenius unit
sphere, then $[u]$ is locally exponentially
attracting in projective space for every
$\eta>2/\lambda$.

It is not locally exponentially attracting for
$0<\eta\le2/\lambda$ wherever the projective map
is defined. Hence its projective selection threshold
and its radial threshold coincide:
\[
 \eta_{\mathrm{sel}}(u)
 =\eta_{\mathrm{edge}}(u)
 =\frac{2}{\lambda}.
\]

In particular, this critical direction has no
interval on which it is simultaneously projectively
exponentially attracting and radially contracting.

\item[(iv)]
Suppose the nondegeneracy assumption in (iii) holds.
For $p>2$, let $s_p=p/(p-1)$ and
\[
 Q_p(Z):=\frac12\|\mathcal AZ\|_{S_{s_p}}^2,
 \qquad
 q_p:=2Q_p\big|_{\{\|u\|_F=1\}}.
\]
For all sufficiently large $p$, there is a unique
critical point $u_p$ of $q_p$ near $u$, with
$u_p\to u$, which is again a nondegenerate strict
local maximum.

Writing $\lambda_p=q_p(u_p)$ and
$\beta_{\min,p}$ for the smallest transverse
eigenvalue of $\nabla^2Q_p(u_p)$, one has
\[
 \lambda_p\longrightarrow\lambda,
 \qquad
 \beta_{\min,p}\longrightarrow0,
\]
and
\[
 \frac{2}{\lambda_p+\beta_{\min,p}}
 \longrightarrow\frac{2}{\lambda},
 \qquad
 \frac{2}{\lambda_p}
 \longrightarrow\frac{2}{\lambda}.
\]

Thus the separation between the selection and
radial thresholds of these nearby critical
directions tends to zero as $p\to\infty$.
\end{enumerate}
\end{proposition}
 
\begin{proof}[Proof of Proposition~\ref{prop:spectral_selection_collapse}]
Write $C=\nabla^2Q_\infty(u)$. Since
$G=\mathcal Au$ has full rank, $Q_\infty$ is $C^2$ near
$u$. Indeed, when $m\ge n$,
\[
  \|Y\|_*=\operatorname{tr}\bigl((Y^\top Y)^{1/2}\bigr)
\]
for full-column-rank $Y$, and the analogous formula
with $YY^\top$ holds when $m<n$. Convexity gives
$C\succeq0$.

\emph{(i).} Take a thin singular value decomposition
\[
  G=U\operatorname{diag}(\sigma_1,\ldots,\sigma_r)V^\top,
  \qquad \sigma_i>0,
\]
and, for $d\in\mathbb R^r$ with $\sum_i d_i=0$, set
\[
  E_d=\mathcal A^{-1}
       \bigl(U\operatorname{diag}(d_1,\ldots,d_r)V^\top\bigr).
\]
For sufficiently small $|t|$, all $\sigma_i+td_i$
are positive. Thus
\[
  \|\mathcal A(u+tE_d)\|_*
   =\sum_{i=1}^r(\sigma_i+td_i)
   =\sum_{i=1}^r\sigma_i,
\]
so $Q_\infty(u+tE_d)$ is constant near $t=0$.
Differentiating twice,
\[
  \langle E_d,CE_d\rangle_F=0.
\]
Since $C\succeq0$, this forces $CE_d=0$.
The map $d\mapsto E_d$ is injective, so these
vectors span an $(r-1)$-dimensional subspace of
$\ker C$.

Criticality on the Frobenius unit sphere and Euler's
identity for the two-homogeneous $Q_\infty$ imply
\[
  \nabla Q_\infty(u)=\lambda u,
  \qquad Cu=\nabla Q_\infty(u)=\lambda u.
\]
By symmetry of $C$,
\[
  \lambda\langle u,E_d\rangle_F
   =\langle Cu,E_d\rangle_F
   =\langle u,CE_d\rangle_F=0.
\]
As $\lambda>0$, each $E_d$ lies in $u^\perp$.
Moreover $C$ preserves $u^\perp$, since
$Cu=\lambda u$. Thus the restriction
$B=C|_{u^\perp}$ is positive semidefinite and
has a kernel of dimension at least $r-1$,
proving (i).

\emph{(ii).} Let
$F_\eta(Z)=Z-\eta\nabla Q_\infty(Z)$. At $u$,
\[
  F_\eta(u)=(1-\eta\lambda)u,
  \qquad DF_\eta(u)=I-\eta C.
\]
If $\eta\ne1/\lambda$, continuity makes $F_\eta$
nonzero in a neighbourhood of $u$, so it induces
a $C^1$ map on projective space near $[u]$.
In local projective coordinates its derivative is
\[
  DP_\eta([u])
   =\left.\frac{I-\eta C}{1-\eta\lambda}
     \right|_{u^\perp}.
\]
Every $E_d$ constructed in (i) is consequently an
eigenvector with multiplier $(1-\eta\lambda)^{-1}$.
For
\[
  0<\eta<2/\lambda,\qquad \eta\ne1/\lambda,
\]
this multiplier has modulus strictly greater than
one, so $[u]$ is linearly unstable. On its
invariant line the multiplier is $1-\eta\lambda$,
whose modulus is less than one. At
$\eta=1/\lambda$ the line is mapped to zero in
one step, and the projective map is undefined
at $[u]$. This proves (ii).

\emph{(iii).} At a critical point, the spherical
Hessian is
\[
  \operatorname{Hess}_{S_F}q_\infty(u)
    =2(B-\lambda I),
  \qquad S_F=\{v:\|v\|_F=1\}.
\]
If $u$ is a nondegenerate strict local maximum,
this Hessian is negative definite. Thus every
eigenvalue $\beta_i$ of $B$ satisfies
$0\le\beta_i<\lambda$. For $\eta>2/\lambda$,
\[
  (1-\eta\lambda)^2-(1-\eta\beta_i)^2
   =\eta(\lambda-\beta_i)
      \bigl(\eta(\lambda+\beta_i)-2\bigr)>0.
\]
All the projective multipliers
\[
  \frac{1-\eta\beta_i}{1-\eta\lambda}
\]
therefore have modulus less than one. The $C^1$
projective map is locally contracting in suitable
coordinates, so $[u]$ is locally exponentially
attracting.

For $0<\eta<2/\lambda$, part (ii) gives an
eigenvalue of modulus greater than one; at
$\eta=2/\lambda$, it equals $-1$, which rules
out local exponential attraction. The projective
map is undefined at $[u]$ for $\eta=1/\lambda$.
Hence attraction starts precisely above
$2/\lambda$, which is also the radial stability
threshold.

\emph{(iv).} Put $s=p/(p-1)$ and write
\[
  Q_s(Z)=\frac12\|\mathcal AZ\|_{S_s}^2.
\]
On a neighbourhood $W$ of $u$, every
$\mathcal AZ$ has full rank. If $m\ge n$, set
$M(Z)=(\mathcal AZ)^\top(\mathcal AZ)\succ0$;
if $m<n$, set
$M(Z)=(\mathcal AZ)(\mathcal AZ)^\top\succ0$.
In either case,
\[
  Q_s(Z)=\frac12
   \bigl(\operatorname{tr}M(Z)^{s/2}\bigr)^{2/s}.
\]
The map $(s,Z)\mapsto Q_s(Z)$ is smooth on a
neighbourhood of $(1,u)$: on positive-definite
matrices,
\[
  M^{s/2}=\exp((s/2)\log M)
\]
depends smoothly on $(s,M)$. Consequently
$Q_s\to Q_\infty$ in $C^2$ on compact subsets
of $W$ as $s\to1$.

Choose a smooth chart $\chi$ of the Frobenius
unit sphere with $\chi(0)=u$. The function
$f(s,y)=2Q_s(\chi(y))$ is smooth near $(1,0)$,
and
\[
  \nabla_y f(1,0)=0,
  \qquad D_y^2f(1,0)
   \quad\text{is negative definite}.
\]
The implicit-function theorem applied to
$\nabla_y f(s,y)=0$ yields, for all $s>1$
sufficiently close to $1$, a unique solution
$y_s$ near zero with $y_s\to0$. The corresponding
$u_s=\chi(y_s)$ is a nondegenerate strict local
maximum of $q_s=2Q_s|_{S_F}$, and $u_s\to u$.

The $C^2$ convergence implies
\[
  \lambda_s=q_s(u_s)\longrightarrow
  q_\infty(u)=\lambda.
\]
It also implies convergence of the transverse
Hessians after identifying the tangent spaces
$u_s^\perp$ with $u^\perp$ through a continuous
orthonormal frame. Eigenvalue continuity
therefore gives
\[
  \beta_{\min,s}\longrightarrow
  \beta_{\min}=0.
\]
For $s=s_p=p/(p-1)$, the selection and radial
thresholds are
\[
  \frac{2}{\lambda_p+\beta_{\min,p}}
  \qquad\text{and}\qquad
  \frac{2}{\lambda_p},
\]
respectively; these formulas apply since $Q_p$
is convex and $u_p$ is a nondegenerate strict
local maximum. Both converge to $2/\lambda$.
In particular,
\[
  \frac{2}{\lambda_p}
  -\frac{2}{\lambda_p+\beta_{\min,p}}
  =
  \frac{2\beta_{\min,p}}
       {\lambda_p(\lambda_p+\beta_{\min,p})}
  \longrightarrow0.
\]
This proves (iv) and the proposition.
\end{proof}

\end{document}